\documentclass[twoside,a4paper,draft,
10pt,reqno,intlimits]{siamltex}
\usepackage{amsmath,color}
\usepackage{amssymb}
\usepackage{amsfonts}
\usepackage[active]{srcltx}
\usepackage
{hyperref}
\usepackage[operator,rose,frak,hyperref]{paper_diening}
\usepackage{graphicx}
\hypersetup{%
  pdftitle={},
  pdfsubject={},
  pdfauthor={},
  pdfkeywords={},
  pdfcreator={pdfLaTeX},
  pdfproducer={pdfLaTeX},
  }
\newtheorem{Theorem}[equation]{Theorem}
\newtheorem{Lemma}[equation]{Lemma}
\newtheorem{Proposition}[equation]{Proposition}
\newtheorem{Corollary}[equation]{Corollary}

\newtheorem{assumption}[equation]{Assumption}
\newtheorem{remark}[equation]{Remark}

\numberwithin{equation}{section}

\newcommand{\meantmp}[2]{#1\langle{#2}#1\rangle}
\newcommand{\mean}[1]{\meantmp{}{#1}}

\newcommand{\Pidiv}{\ensuremath{\Pi^{\divergence}_h}}

\newcommand{\Rnn}{{\setR^{n \times n}}}
\newcommand{\para}{{\delta}}

\newcommand{\PiY}{\Pi^Y_h}

\newcommand{\parameter}{\delta}

\newcommand{\bfem}{\mbox{\bf e}^m}

\newcommand{\bfutm}{\mbox{\bf u}(t_m)}
\newcommand{\bfum}{\mbox{\bf u}^m}
\newcommand{\bfumh}{\mbox{\bf u}^m_h}
\newcommand{\ksum}{{k\sum_{m=0}^M}}
\newcommand{\ksumN}{{k\sum_{m=1}^N}}

\begin{document}
\title
{Optimal error estimate for semi-implicit space-time discretization
  for the equations describing incompressible generalized Newtonian
  fluids}
\author{Luigi C.\ Berselli\thanks{Dipartimento di
    Matematica, Universit{\`a} di Pisa, Via
    F.~Buonarroti 1/c, I-56127 Pisa, ITALY.
    ({berselli@dma.unipi.it})}
\and Lars Diening\footnotemark[2] 
\and Michael  R\r u\v zi\v cka{}\footnotemark[3] }\date{}
\maketitle
\renewcommand{\thefootnote}{\fnsymbol{footnote}} 
\footnotetext[2]{Institute of Mathematics, LMU Munich,
  Theresienstr.~39, D-80333, Munich, GERMANY,
  (diening@mathematik.uni-muenchen.de) }

\footnotetext[3]{Institute of Applied Mathematics,
  Albert-Ludwigs-University Freiburg, Eckerstr.~1, D-79104 Freiburg,
  GERMANY. (rose@mathematik.uni-freiburg.de)}

\begin{abstract}
  In this paper we study the numerical error arising in the space-time
  approximation of unsteady generalized Newtonian fluids which possess
  a stress-tensor with $(p.\delta)$-structure.  A semi-implicit
  time-discretization scheme coupled with conforming inf-sup stable finite
  element space discretization is analyzed. The main result, which improves
  previous suboptimal estimates as those in [A.~Prohl, and
  M.~\Ruzicka, \textit{SIAM J.  Numer. Anal.,} 39 (2001),
  pp.~214--249] is the optimal $\mathcal{O}(k+h)$ error-estimate valid
  in the range $p\in (3/2,2]$, where $k$ and $h$ are the time-step and
  the mesh-size, respectively.  Our results hold in three-dimensional
  domains (with periodic boundary conditions) and are uniform with
  respect to the degeneracy parameter $\para \in [0,\delta_0]$ of the
  extra stress tensor.
\end{abstract}
\date{\small \today}

\maketitle
\section{Introduction}
We study the (full) space-time discretization of a homogeneous (for
simplicity the density $\rho$ is set equal to $1$), unsteady, and
incompressible fluid with shear-dependent viscosity, governed by the
following system of partial differential equations
\begin{equation}
  \label{eq:pfluid}\tag{$\text{NS}_p$}
  \begin{aligned}
    \bfu_t-\divo \bfS(\bfD\bfu)+ [\nabla \bfu] 
    \bfu+\nabla\pi&= \bff\qquad&&\text{in } I\times \Omega,
    \\
    \divo\bfu&=0\qquad&&\text{in } I\times \Omega, 
    \\
    \bfu(0)&=\bfu_0\qquad&&\text{in }  \Omega,
  \end{aligned}
\end{equation}
where the vector field $\bfu=(u_1,u_2,u_3)$ is the velocity, $\bfS$ is
the extra stress tensor, the scalar $\pi$ is the kinematic pressure,
the vector $\bff=(f_1,f_2,f_3)$ is the external body force, and
$\bfu_0$ is the initial velocity. We use the notation $([\nabla
\bfu]\bfu)_i = \sum_{j=1}^3 u_j \partial_j u_i$, ${i=1,2,3}$, for the
convective term, while $\bfD \bfu := \tfrac 12 (\nabla\bfu+\nabla
\bfu^\top )$ denotes the symmetric part of the gradient $\nabla \bfu$.
Throughout the paper we shall assume that
$\Omega=(0,2\pi)^3\subset\setR^3$ and we endow the problem with space
periodic boundary conditions. As explained in~\cite{bdr-7-5-num,bdr-7-5},
this assumption simplifies the problem, allows us to prove suitable
regularity results for both the continuous and the time-discrete
problems, so we can concentrate on the difficulties that
arise from the structure of the extra stress tensor.  As usual
$I=[0,T]$ denotes some non-vanishing time interval.

The most standard example of power-law like extra stress tensors in the
class under consideration (cf.~Assumption~\ref{ass:1}) is, for
$p\in(1,\infty)$,
\begin{equation*}
  \bfS(\bfD \bfu) =\mu\, (\para+\abs{\bfD\bfu})^{p-2}\bfD\bfu,
\end{equation*} 
where $\mu>0$ and $\para\geq0$ are given constants. The literature on
this subject is very extensive (cf.~\cite{bdr-phi-stokes, bdr-7-5-num}
and the discussion therein). Based on the results
in~\cite{bdr-phi-stokes,bdr-7-5-num} we find here a suitable setting
for the choice of the finite-element space-discretization and for the
semi-implicit Euler scheme for time advancing, in order to show a 
convergence result, which is optimal apart from the $h$-$k$ coupling.
Previous results in this direction have been proved in~\cite{pr} even
if the lack of available precise regularity results lead to
non-optimal results. In fact an error
$\mathcal{O}(h+k)^{\frac{5p-6}{2p}}$ for
$p\in\big]\frac{3+\sqrt{29}}{5},2[$ was obtained in the space-periodic
three-dimensional case in~\cite{pr} in the case of conforming and
non-conforming finite elements.

Here, we combine the optimal estimates for the time-discretization
from~\cite{bdr-7-5-num} with those for the stationary problem (without
convective term) from~\cite{bdr-phi-stokes}, and also the results for
parabolic systems in~\cite{der}, to produce the optimal
$\mathcal{O}(k+h)$ order of convergence for a natural distance, see
Theorem~\ref{thm:main_thm} for the precise statement of the result.

\vspace{.2cm}

\textbf{Plan of the paper:} In Section~\ref{sec:preliminaries} we
introduce the notation, the main hypotheses on the stress-tensor, and
the properties of the numerical methods we consider. We also recall
some technical results from previous papers which we will need later
on. The proof of the main estimate on the numerical error is then
postponed to Section~\ref{sec:proof-main-result}.
\section{Notation and preliminaries}
\label{sec:preliminaries}
In this section we introduce the notation we will use and we 
also recall some technical results which will be needed in the proof
of the main convergence result.
\subsection{Function spaces}
\label{sec:notation} 
We use $c, C$ to denote generic constants, which may change from line
to line, but are not depending on the crucial quantities. Moreover we
write $f\sim g$ if and only if there exists constants $c,C>0$ such
that $c\, f \le g\le C\, f$. Given a normed space $X$ we denote its
topological dual space by $X^*$. We denote by $\abs{M}$ the
$n$-dimensional Lebesgue measure of a measurable set $M$. The mean
value of a locally integrable function $f$ over a measurable set $M
\subset \Omega$ is denoted by $\mean{f}_M:= \dashint_M f \, dx =\frac
1 {|M|}\int_M f \, dx$. Moreover, we use the notation
$\skp{f}{g}:=\int_\Omega f g\, dx$, whenever the right-hand side is
well defined.

We will use the customary Lebesgue spaces $L^p(\Omega)$ and Sobolev
spaces $W^{k,p}(\Omega)$, where $\Omega =(0,2\pi)^3$ and
periodic conditions are enforced. As usual $p':=\frac{p}{p-1}$. In
addition, $W^{k,p}_{\divo}(\Omega)$ denotes the subspace of (vector
valued) functions with vanishing divergence. We will denote by
$\norm{\,.\,}_p$ the norm in $L^p(\Omega)$ and, in the case of zero
mean value, we equip $W^{1,p}(\Omega)$ (based on the \Poincare{}
Lemma) with the gradient norm $\norm{\nabla \,.\,}_p$.

For the time-discretization, given $T>0$ and $M\in \setN$, we define
the time-step size as $k:=T/M>0$, with the corresponding net
$I^M:=\{t_m\}_{m=0}^M$, $t_m:=m\,k$, and we define the
finite-differences backward approximation for the time derivative as:
\begin{equation*}
  d_t \bu^m:=\frac{\bu^m-\bu^{m-1}}{k}.
\end{equation*}
To deal with discrete problems we shall use the discrete spaces
$l^p(I^M;X)$ consisting of $X$-valued sequences $\{a_m\}_{m=0}^M$,
endowed with the norm
\begin{equation*}
  \|a_m\|_{l^p(I^M;X)}:=
  \left\{\begin{aligned}
      &\left(\ksum\|a_m\|_X^p\right)^{1/p}\quad &&\text{if }1\leq
      p<\infty\,,
      \\
      &\max_{0\leq m\leq M}\|a_m\|_X &&\text{if }p=\infty\,.
    \end{aligned}
  \right.
\end{equation*}

For the space discretization, $\mathcal{T}_h$ denotes a family of
shape-regular, conformal triangulations, consisting of
three-dimensional simplices $K$. We denote by $h_K$ the diameter of
$K$ and by $\rho_K$ the supremum of the diameters of inscribed balls.
We assume that $\mathcal{T}_h$ is non-degenerate, i.e., $\max _{K \in
  \mathcal{T}_h} \frac {h_k}{\rho_K}\le \gamma_0$.  The global
mesh-size $h$ is defined by $h:=\max _{K \in \mathcal T_h}h_K$.  Let
$S_K$ denote the neighborhood of~$K$, i.e., $S_K$ is the union of all
simplices of~$\mathcal{T}$ touching~$K$.  By the assumptions we obtain
that $|S_K|\sim |K|$ and that the number of patches $S_K$ to which a
simplex belongs is bounded uniformly with respect to $h$ and $K$.

The function spaces which we will use are the following 
\begin{alignat*}{2}
  X &:= \big(W^{1,p}(\Omega)\big)^n, \qquad & V &:=
  \bigg\{\bu\in X :\ \dashint_\Omega    \bu\,dx=0 \bigg \},
  \\
  Y &:= L^{p'}(\Omega)\,,\qquad & Q &:=
  L^{p'}_0(\Omega):=\biggset{f\in Y\,:\, \dashint_\Omega
    f\,dx=0 }.
\end{alignat*}
In the finite element analysis, we denote by ${\mathcal
  P}_m(\mathcal{T}_h)$,  $m \in \setN_0$, the space of scalar or
vector-valued continuous functions, which are polynomials of degree at
most $m$ on each simplex $K \in \mathcal{T}_h$.  Given a triangulation
of $\Omega$ with the above properties and given $k,m \in \setN$ we
denote by $X_h \subset {\mathcal P}_m(\mathcal{T}_h)$ and $Y_h\subset
{\mathcal P}_k(\mathcal{T}_h)$ appropriate conforming finite element
spaces defined on $\mathcal{T}_h$, i.e., $X_h$, $Y_h$ satisfy $X_h
\subset X$ and $Y_h \subset Y$. Moreover, we set $V_h := X_h \cap V$
and $Q_h:= Y_h \cap Q$, while $\skp{f}{g}_h:= \sum_{K \in
  \mathcal{T}_h} \int_K{f}{g}\, dx$ denotes the inner
product in the appropriate spaces.

For the error estimates it is crucial to have projection operators
well-behaving in terms of the natural norms. As in~\cite{bdr-phi-stokes} we
make the following assumptions on the projection operators associated
with these spaces.
\begin{assumption}
  \label{ass:proj-div}
  We assume that $\mathcal{P}_1(\mathcal{T}_h) \subset X_h$ and there
  exists a linear projection operator $\Pidiv\,:\, X \to X_h$ which
  \begin{enumerate}
  \item preserves divergence in the $Y_h^*$-sense, i.e.,
    \begin{align*}
      \skp{\divergence \bfw}{\eta_h} &= \skp{\divergence \Pidiv
        \bfw}{\eta_h} \qquad \forall\, \bfw \in X,\; \forall\, \eta_h \in
      Y_h\,;
    \end{align*}
  \item preserves periodic conditions, i.e. $\Pidiv(X) \subset X_h$;
  \item is locally $W^{1,1}$-continuous in the sense that
    \begin{align*}
      \dashint_K \abs{\Pidiv\bfw}\,dx &\leq c \dashint_{S_K}\!
      \abs{\bfw}\,dx + c \dashint_{S_K}\! h_K \abs{\nabla \bfw}\,dx
      \quad \forall\, \bfw \in X,\; \forall\, K \in \mathcal{T}_h.
    \end{align*}
  \end{enumerate}
\end{assumption}
%
%
\begin{assumption}
  \label{ass:PiY}
  We assume that $Y_h$ contains the constant functions, i.e. that
  $\setR \subset Y_h$, and that there exists a linear projection
  operator $\PiY\,:\, Y \to Y_h$ which is locally $L^1$-continuous in
  the sense that
  \begin{align*}
    \dashint_K \abs{\PiY q}\,dx &\leq c\, \dashint_{S_K} \abs{q}\,dx
    \qquad \forall\, q \in Y,\; \forall\, K \in \mathcal{T}_h.
  \end{align*}
\end{assumption}%

For a discussion and consequences of these assumptions  we refer
to~\cite{bdr-phi-stokes}.  In particular we will need the following results:

\begin{Proposition}
  \label{thm:ocont}
  Let $r \in (1,\infty)$ and let $\Pidiv$ satisfy
  Assumption~\ref{ass:proj-div}. Then $\Pidiv$ has the following local
  continuity property
  \begin{align*}
    \int_K  \abs{\nabla \Pidiv \bfw}^r \,dx
    \leq c\, \int_{S_K} \abs{\nabla \bfw}^r
    \,dx
  \end{align*}
  and the following local approximation property
  \begin{align*}
    \int_K \abs{\bfw -\Pidiv \bfw}^r \,dx \leq c\, \int_{S_K}
    h_K^r \abs{\nabla \bfw}^r\,dx,
  \end{align*}
  for all $K \in \mathcal{T}_h$ and $\bfw \in (W^{1,r}(\Omega))^n$.
  The constant $c$ depends only on $r$ and on the
  non-degeneracy constant $\gamma_0$ of the triangulation
  $\mathcal{T}_h$.
\end{Proposition}
\begin{proof}
  This is special case of  Thm.~3.5 in \cite{bdr-phi-stokes}. 
\end{proof}

\begin{Proposition}
  \label{lem:PiYstab}
  Let $r \in (1,\infty)$ and let $\PiY$ satisfy
  Assumption~\ref{ass:PiY}.  Then for all $K \in \mathcal{T}_h$ and $q
  \in L^r(\Omega)$ we have
  \begin{align*}
    \int_K  \abs{\PiY q}^r \,dx &\leq c\, \int_{S_K} 
    \abs{q}^r\,dx.
  \end{align*}
  Moreover, for all $K \in \mathcal{T}_h$ and $q \in W^{1,r}(\Omega)$
  we have
  \begin{align*}
    \int_K \abs{q- \PiY q}^r\,dx &\leq c\, \int_{S_K}
     h_K^r \abs{\nabla q}^r \,dx.
  \end{align*}
  The constants depend only on $r$ and on $\gamma_0$.
\end{Proposition}
\begin{proof}
  This is special case of  Lemma~5.3 in \cite{bdr-phi-stokes}. 
\end{proof}

\begin{remark}\label{rem:int}
  {\rm
    By summing over $K \in \mathcal T_h$ one can easily get global
    analogues of the statements in the above Propositions.
  }
\end{remark}

As usual, to have a stable space-discretization, we use the following
tri-linear form in the weak formulation of (space) discrete problems
\begin{equation*}
  b(\bu,\bv,\bw):=
  \frac{1}{2}\big[\skp{[\nabla\bv]\bu}{\bw}_h-\skp{[\nabla\bw]\bu}{\bv}_h\big],
\end{equation*}
observing that for periodic divergence-free functions (in the
continuous sense) it holds $b(\bu,\bv,\bw):=
\skp{[\nabla\bv]\bu}{\bw}$.
\subsection{Basic properties of the extra stress tensor} 
\label{sec:stress_tensor}
For a second-order tensor $\bfA \in \setR^{n \times n} $ we denote its
symmetric part by $\bA^\sym:= \frac 12 (\bA +\bA^\top) \in \setR_\sym
^{n \times n}:= \set {\bfA \in \setR^{n \times n} \,|\, \bA
  =\bA^\top}$. The scalar product between two tensors $\bA, \bB$ is
denoted by $\bA\cdot \bB$, and we use the notation $\abs{\bA}^2=\bA
\cdot \bA ^\top$. We assume that the extra stress tensor $\bS$ has
$(p,\para)$-structure, which will be defined now.  A detailed
discussion and full proofs of the following results can be found
in~\cite{DieE08,dr-nafsa}.
\begin{assumption}
  \label{ass:1}
  We assume that the extra stress tensor $\bS\colon \setR^{n \times n}
  \to \setR^{n \times n}_\sym $ belongs to $C^0(\setR^{n \times
    n},\setR^{n \times n}_\sym )\cap C^1(\setR^{n \times n}\setminus
  \{\bfzero\}, \setR^{n \times n}_\sym ) $, satisfies $\bS(\bA) =
  \bS\big (\bA^\sym \big )$, and $\bS(\mathbf 0)=\mathbf 0$. Moreover,
  we assume that the tensor $\bS$ has {\rm $(p,\para)$-structure},
  i.e., there exist $p \in (1, \infty)$, $\para\in [0,\infty)$, and
  constants $C_0, C_1 >0$ such that
   \begin{subequations}
     \label{eq:ass_S}
     \begin{align}
       \sum\nolimits_{i,j,k,l=1}^n \partial_{kl} S_{ij} (\bA) C_{ij}C_{kl} &\ge C_0 \big
       (\para +|\bA^\sym|\big )^{{p-2}} |\bC^\sym |^2,\label{1.4b}
       \\
       \big |\partial_{kl} S_{ij}({\bA})\big | &\le C_1 \big (\para
       +|\bA^\sym|\big )^{{p-2}},\label{1.5b}
     \end{align}
   \end{subequations}
   are satisfied for all $\bA,\bC \in \setR^{n\times n} $ with
   $\bA^\sym \neq \bfzero$ and all $i,j,k,l=1,\ldots, n$.  The
   constants $C_0$, $C_1$, and $p$ are called the {\em
     characteristics} of $\bfS$.
\end{assumption}
\begin{remark}
  {\rm 
    We would like to emphasize that, if not otherwise 
    stated, the constants in the paper depend only on the
    characteristics of $\bfS$ but are independent of $\delta\geq 0$. 
  }
\end{remark}

Another important set of tools are the {\rm shifted N-functions}
$\set{\phi_a}_{a \ge 0}$, cf.~\cite{DieE08,DieK08,dr-nafsa}. To this
end we define for $t\geq0 $ a special N-function $\phi$ by
\begin{align} 
  \label{eq:5} 
  \varphi(t):= \int _0^t \varphi'(s)\, ds\qquad\text{with}\quad
  \varphi'(t) := (\delta +t)^{p-2} t\,.
\end{align}
Thus we can replace in the right-hand side of~\eqref{eq:ass_S} the
expression $C_i \big
(\para +|\bA^\sym|\big )^{{p-2}} $ by $\widetilde C_i\,\varphi'' \big
(|\bA^\sym|\big )$, $i=0,1$. Next, the shifted functions are defined
for $t\geq0$ by
\begin{align*}
  \varphi_a(t):= \int _0^t \varphi_a'(s)\, ds\qquad\text{with }\quad
  \phi'_a(t):=\phi'(a+t)\frac {t}{a+t}.
\end{align*}
For the $(p,\delta)$-structure we have that $\phi_a(t) \sim
(\delta+a+t)^{p-2} t^2$ and also $(\phi_a)^*(t) \sim ((\delta+a)^{p-1}
+ t)^{p'-2} t^2$, where the $*$-superscript denotes the complementary
function\footnote[2]{For a N-function $\psi $ the complementary
  N-function $\psi^*$ is defined by $\psi^*(t):= \sup_{s \ge
    0}(st-\psi(s))$.}. We will use also Young's inequality: For all
$\varepsilon >0$ there exists $c_\epsilon>0 $, such that for all
$s,t,a\geq 0$ it holds
\begin{align}
  \label{eq:ineq:young}
  \begin{split}
    ts&\leq \epsilon \, \phi_a(t)+ c_\epsilon \,(\phi_a)^*(s)\,,
    \\
    t\, \phi_a'(s) + \phi_a'(t)\, s &\le \epsilon \, \phi_a(t)+ c_\epsilon
    \,\phi_a(s).
  \end{split}
\end{align}

Closely related to the extra stress tensor $\bS$ with
$(p,\delta)$-structure is the function $\bF\colon\setR^{n \times n}
\to \setR^{n \times n}_\sym$ defined through
\begin{align}
  \label{eq:def_F}
  \bF(\bA):= \big (\para+\abs{\bA^\sym} \big )^{\frac
    {p-2}{2}}{\bA^\sym } \,.
\end{align}
The main calculations of the paper can be performed by recalling the
following lemma, which establishes the connection between $\bfS$,
$\bfF$, and $\set{\phi_a}_{a \geq 0}$ (cf.~\cite{DieE08,dr-nafsa}).
\begin{Lemma}
  \label{lem:hammer}
  Let $\bfS$ satisfy Assumption~\ref{ass:1}, let $\phi$ be defined
  in~\eqref{eq:5}, and let $\bfF$ be defined in~\eqref{eq:def_F}.
  Then
  \begin{subequations}
    \label{eq:hammer}
    \begin{align}
      \label{eq:hammera}
      \big({\bfS}(\bfP) - {\bfS}(\bfQ)\big) \cdot \big(\bfP-\bfQ \big)
      &\sim \bigabs{ \bfF(\bfP) - \bfF(\bfQ)}^2
      \\
      \label{eq:hammerb}
      &\sim \phi_{\abs{\bfP^\sym}}(\abs{\bfP^\sym - \bfQ^\sym})
      \\
      \label{eq:hammerc}
      &\sim \phi''\big( \abs{\bfP^\sym} + \abs{\bfQ^\sym}
      \big)\abs{\bfP^\sym - \bfQ^\sym}^2
      \\
      \intertext{uniformly in $\bfP, \bfQ \in \setR^{n \times n}$.
        Moreover,  uniformly in $\bfQ \in \setR^{n \times n}$,} 
      \label{eq:hammerd}
      \bfS(\bfQ) \cdot \bfQ \sim \abs{\bfF(\bfQ)}^2 &\sim
      \phi(\abs{\bfQ^\sym}).
    \end{align}
  \end{subequations}
  The constants depend only on the characteristics of $\bfS$.
\end{Lemma} 
%

Moreover, we observe that
\begin{align}
  \label{eq:hammere}  
  \abs{\bfS(\bfP) - \bfS(\bfQ)} &\sim
  \phi'_{\abs{\bfP^\sym}}\big(\abs{\bfP^\sym -
    \bfQ^\sym}\big)\qquad\forall\,\bfP, \bfQ \in \Rnn,
\end{align}
which allows us  to introduce a ``\textit{Natural distance}'' since by
the previous lemma we have, for all sufficiently smooth vector fields
$\bfu$ and $\bfw$,
  \begin{align*}
    \skp{\bfS(\bD \bfu) \!-\! \bfS(\bD\bfw)}{\bD\bfu \!-\! \bD \bfw}
    &\sim 
    \norm{\bfF(\bD\bfu) \!-\! \bfF(\bD\bfw)}_2^2 \,\sim \int_\Omega\!
    \phi_{\abs{\bD\bfu}}(\abs{\bD\bfu \!-\! \bD\bfw}) \,dx,
  \end{align*}
  and again the constants depend only on the characteristics of
  $\bfS$.
  
  In view of Lemma~\ref{lem:hammer} one can deduce many useful
  properties of the natural distance and of the quantities $\bF$,
  $\bS$ from the corresponding properties of the shifted N-functions
  $\{\phi_a\}$. For example the following important estimates follow
  directly from~\eqref{eq:hammere}, Young's
  inequality~\eqref{eq:ineq:young}, and~\eqref{eq:hammer}.

\begin{Lemma}
  \label{lem:quasinormtrick}
  For all $\epsilon>0$, there exist a constant $c_\epsilon>0$
  (depending only on $\epsilon>0$ and on the characteristics of
  $\bfS$) such that for all sufficiently smooth vector fields $\bu$,
  $\bv$, and $\bw$ we have
  \begin{align*}
    &\skp{\bfS(\bD\bfu) - \bfS(\bD\bfv)}{\bD\bfw - \bD \bfv}
\leq \epsilon\, \norm{\bfF(\bD\bfu) - \bfF(\bD\bfv)}_2^2
       +c_\epsilon\,  \norm{\bfF(\bD\bfw) - \bfF(\bD\bfv)}_2^2\,.
  \end{align*}
\end{Lemma}
 \subsection{Some technical preliminary results}
 We recall some regularity results for fluids with shear dependent
 viscosities (both continuous and time-discrete) and some convergence
 results we will need in the sequel.

 First, we recall that for the continuous problem~\eqref{eq:pfluid} we
 have the following existence and uniqueness result for strong
 solutions (cf.~\cite[Thm.~5.1]{bdr-7-5}).
\begin{theorem} 
  \label{thm_existence_strong}
  Let $\bS$ satisfy Assumption~\ref{ass:1} with $\frac{7}{5} <p\leq 2$
  and $\parameter \in [0,\parameter_0]$ with $\parameter_0>0$. Assume
  that ${\bff}\in {L^\infty(I; W^{1,2}(\Omega))} \cap
  {W^{1,2}(I;L^2(\Omega))}$ and also
  ${\bfu_0}\in{W_{\divergence}^{2,2}(\Omega)}$, $\skp{\bu_0}{1}=0$,
  and $\divo\bS(\bD\bu_0)\in L^2(\Omega)$.  Then, there exist a time
  $T'>0$ and a constant $c_0>0$, both depending on $(\parameter_0,p,C_0,\ff,
  \bfu_0,T,\Omega)$ but independent of $\parameter$, such that the
  system~\eqref{eq:pfluid} has a unique strong solution $\bfu\in
  L^p(I';W^{1,p}_{\divergence} (\Omega))$, $I' = [0,T']$ such that
  \begin{align}
    \label{ineq:locally_strong_for_7_over_5_reg1}
    \begin{split}
     \hspace*{-1mm} \norm{ \bfu_t}_{L^\infty(I'; L^2(\Omega))} +
      &\norm{\bF(\bD\bfu)}_{W^{1,2}(I'\times\Omega)} +
      \norm{\bF(\bD\bfu)}_{L^{2\frac{5 p-6}{2-p}}(I';W^{1,2}(\Omega))}
      \leq c_0\,.
    \end{split}
  \end{align}
  In particular this implies, uniformly in $\delta\in[0,\delta_{0}]$,
  \begin{subequations}
    \label{sob}
    \begin{align}
      &\bfu\in
      \setL^{\frac{p(5p-6)}{2-p}}(I';\setW^{2,\frac{3p}{p+1}}(\Omega))
      \cap  C(I';\setW^{1,r}(\Omega))\qquad \text{for }
      1\leq  r<6(p-1),\label{boundedness}
      \\
      &\bfu_t \in \setL^\infty(I';\setL^2(\Omega))\cap
      \setL^{\frac{p(5p-6)}{(3p-2)(p-1)}}(I';W^{1,\frac{3p}{p+1}}(\Omega))\,.
      \label{time_derivative}
    \end{align}
  \end{subequations}
\end{theorem}%
The above theorem, whose proof employs in a substantial manner the
hypothesis of space-periodicity, has been used to prove the following
optimal convergence result for the numerical error with respect to a
semi-implicit time discretization (cf.~\cite[Thm~1.1, 4.1]{bdr-7-5-num}).
\begin{theorem} 
  \label{thm_regularity_discrete}
  Let $\bS$ satisfy Assumption~\ref{ass:1} with
  $p\in\big(\frac{3}{2},2]$ and $\para \in [0,\para_0]$, where
  $\para_0>0$ is an arbitrary number. Let $ {\bff} \in {C(I;
    W^{1,2}(\Omega))} \cap {W^{1,2}(I; L^2(\Omega))}$, where
  $I=[0,T]$, for some $T>0$, and let ${\bfu_0} \in
  {W_{\divergence}^{2,2}(\Omega)}$ with $\divo \bS(\bD \bu_0) \in
  L^2(\Omega)$ be given. Let $\bfu $ be a strong
  solution of the (continuous) problem~\eqref{eq:pfluid} satisfying
  \begin{align}
    \label{eq:reg1}
    \begin{split}
      \norm{ \bfu_t}_{L^\infty(I; L^2(\Omega))} +
      \norm{\bF(\bD\bfu)}_{W^{1,2}(I\times\Omega)} +
      \norm{\bF(\bD\bfu)}_{L^{2\frac{5 p-6}{2-p}}(I;W^{1,2}(\Omega))}
      &\leq c_{1}.
    \end{split}
  \end{align}
  Then, there exists $k_{0}>0$ such that for $k\in(0,k_{0})$ the
  unique time-discrete solution $\bfum$ of the semi-implicit time-discrete
  iterative scheme
  \begin{equation}
    \label{m_iterate} 
  \tag{$\text{NS}_p^k$}
  \begin{aligned}
    d_t \bfum-\divo \bS(\bD\bfum)+[\nabla
    \bfum]\bu^{m-1}+\nabla\pi^m&=\ff(t_m)\qquad&&\text{in }\Omega,
    \\
    \divo\bfum&=0&&\text{in }\Omega,
  \end{aligned}
\end{equation}
(endowed with periodic boundary conditions) satisfies the error
estimate
\begin{equation*}
    \max_{0\leq m\leq M}\|\bfutm-\bfum\|^2_2+
    \ksum\|\bF(\bD\bfutm)-\bF(\bD\bfum)\|^2_{2}\leq c\,
    k^2,  
  \end{equation*}
  where the constants ${k}_{0}$ and $c$ depend on $c_{1}$ and on the
  characteristics of $\bS$, but are independent of $\parameter$.
  Moreover, for each $1\leq r<6(p-1)$, it holds
\begin{subequations}
  \label{eq:regularity_discrete_problem}
  \begin{align}
    &\bfum\in
    l^{\frac{p(5p-6)}{2-p}}(I^M;W^{2,\frac{3p}{p+1}}(\Omega)) \cap
    l^\infty(I^M;W^{1,r}(\Omega)),
    \\
    &d_t\bfum\in l^\infty(I^M;L^{2}(\Omega))\cap
    l^{\frac{p(5p-6)}{(3p-2)(p-1)}}
    (I^M;W^{1,\frac{3p}{p+1}}(\Omega)).
  \end{align}
\end{subequations}
\end{theorem}%

We observe that by parabolic interpolation, (cf.~\cite[Rem.~2.7]{bdr-7-5-num}) it also follows
that
\begin{equation*}
  d_t\bfum\in l^{\frac{11p-12}{3(p-1)}}(I^M;L^{\frac{11p-12}{3(p-1)}}(\Omega)),  
\end{equation*}
and consequently
\begin{equation}
  \label{eq:regularity_time-discrete-derivative}
  d_t\bfum  \in
  l^{\frac{p}{p-1}}(I^M;L^{\frac{p}{p-1}}(\Omega))=l^{p'}(I^M;L^{p'}(\Omega))\quad
  \text{ if }\ p>\frac{3}{2}. 
\end{equation}
The latter property will have a relevant role to estimate in the error
equation the term involving the discrete pressure.

\bigskip

One main tool in the sequel will be also the following generalized
Gronwall lemma, which is a minor variation of that proved in great
detail in~\cite[Lemma~3.3]{bdr-7-5-num}.
\begin{lemma}
  \label{lem:discrete_Gronwall_lemma}
  Let $1<p\leq 2$ and let be given two non-negative sequences
  $\{a_m\}_m$ and $\{b_m\}_m$, and two sequences $\{r_m\}_m$ and
  $\{s_m\}_m$ for which there exists $\gamma_0>0$ such that for all
  $0<h<1/\sqrt{\gamma_0}$:
  \begin{equation}
    \label{eq:r_m}
    \begin{aligned}
      &a_0^{2}\leq \gamma_0\,h^{2}, \quad
      b_0^{2}\leq \gamma_0\,h^{2},\quad\ksum r_m^2\leq
      \gamma_0\,h^2,\quad\text{and}\quad\ksum s_m^2\leq \gamma_0\,h^2. 
    \end{aligned}
  \end{equation}
  Further, let there exist constants $\gamma_1,\,\gamma_2,\gamma_3>0$,
  $ \Lambda>0$, and some $0<\theta\leq 1$ such that for some $\lambda
  \in [0,\Lambda]$ the following two inequalities are satisfied for
  all $m\geq1$:
  \begin{align}
    \label{discrete_Gronwall_bis}
    &d_t a_m^2+\gamma_1(\lambda+b_m)^{p-2}b_m^2\leq b_m r_m+
    \gamma_2 b_{m-1}b_m+s_m^2,
    \\
    \label{discrete_Gronwall_ter}
    &d_t a_m^2+\gamma_1(\lambda+b_m)^{p-2}b_m^2\leq b_m r_m+\gamma_3
    b_{m}b_{m-1}^{1-\theta}a_m^\theta+s_m^2 .  
  \end{align}
  %
  %
  Then, there exist $\overline{k},\,\overline{\gamma_0}>0$ such that if
  $h^{2}<\overline{\gamma_0}\,k$ and if~\eqref{discrete_Gronwall_bis},
  \eqref{discrete_Gronwall_ter} hold for $0<k<\overline{k}\leq1$, then
  there exist $\gamma_4,\,\gamma_5>0$, independent of $\lambda$, such
  that
  \begin{align}
    &\max_{1\leq m\leq M} b_m\leq 1,
    \\
    \label{Gronwall_discrete2}
    &\max_{1\leq m\leq M} a_m^2+{\gamma_1(\lambda+\Lambda)^{p-2}}
    \ksumN b_m^2\leq \gamma_4\,h^{2}\,\textrm{exp}(2 \gamma_{5}k\,M).
  \end{align}
\end{lemma}
\begin{proof}
  The proof of this result is a simple adaption of that
  of~\cite[Lemma~3.3]{bdr-7-5-num}.  Nevertheless we report the main
  changes needed to accomplish the proof.  In particular, we will use
  it for $a_m:=\norm{\bfum-\bfum_h}_{2}$ and
  $b_m:=\norm{\bD\bfum-\bD\bfum_h}_{p}$.

  The proof goes by induction on $1\leq N\le M$. Since in the
  inequality~\eqref{discrete_Gronwall_ter} the term $b_{m-1}$ is
  present and since contrary to Ref.~\cite {bdr-7-5-num}
  $a_{0},b_{0}\not=0$, some care has to be taken to start the
  induction argument.  The most important part of the proof is that of
  showing that $b_{m}\leq1$, because then the
  estimate~\eqref{Gronwall_discrete2} will follow by applying the
  classical discrete Gronwall lemma.  We will use the same argument to
  check as starting inductive step that~\eqref{Gronwall_discrete2} is
  satisfied for $N=1$, as well as to show that if
  inequality~\eqref{Gronwall_discrete2} is satisfied for a given
  $N\ge 1$, then holds true also for $N+1$.

  Let us suppose \textit{per absurdum} that $b_N>1$, while
  $b_{m}\leq1$ for $m<N$. We multiply~\eqref{discrete_Gronwall_bis} by
  $k$ and we sum over $m$, for $m=1,\dots,N$. It readily follows that:
  \begin{equation*}
    \begin{aligned}
      & a_N^2+\gamma_1\ksumN(\lambda+b_m)^{p-2}b_m^2\leq
      \\
      & \leq
      a_{0}^{2}+\frac{\gamma_1}{2}\ksumN(\lambda+b_m)^{p-2}b_m^2
      +\frac{1}{\gamma_1}\ksumN 
      (\lambda+b_m)^{2-p}(r_m^2+\gamma_2^2b_{m-1}^2)+ \ksumN s_m^2.
    \end{aligned}
  \end{equation*}
  We absorb the second term from the right-hand side in the left-hand
  side and we observe that $ (\lambda+b_m)^{2-p}\leq
  (\lambda+b_N)^{2-p}\leq(\lambda+b_N)^{2(2-p)}$, regardless of the
  value of $\lambda\ge 0$.  Neglecting all terms on the left-hand
  side, except the one with $m=N$, and dividing both sides by
  $\frac{\gamma_1}{2}k(\lambda+b_N)^{p-2}\not=0$ we get,
  \begin{equation}
    \label{eq:inductive}
    b_N^2\leq \frac{2(\lambda+b_N)^{2(2-p)}}{k \gamma_1}
    \left[a_{0}^{2}
      +\frac{1}{\gamma_{1}}\ksumN
      (r_m^2+\gamma_2^2b_{m-1}^2)+\ksumN s_m^2\right].
  \end{equation}
  Now, if we are dealing with the initial step $N=1$, we have on the
  right-hand side of~\eqref{eq:inductive} a term containing
  $b_{0}^{2}$ on which we need to show that it 
  satisfies~\eqref{Gronwall_discrete2}. The hypothesis~\eqref{eq:r_m}
  and the restriction on $h$ imply that $a_{0}\leq \gamma_0\,h^{2}$ and
  $b_{0}\leq 1$. We also need to satisfy the same
  estimate~\eqref{Gronwall_discrete2}  when $m=0$, namely:
  \begin{equation*}
    \gamma_{1}(\lambda+\Lambda)^{p-2}  k\,      b_0^2\leq
    \gamma_4\,h^{2}\,\textrm{exp}(2 \gamma_{5}k\,M).
  \end{equation*}
  Since $k\leq1$, and given $\gamma_0>0$, it is enough to choose $\gamma_{4}>0$
  large enough such that the following inequality is satisfied
  \begin{equation}
    \label{eq:gamma4}
    \gamma_0\leq\min\left\{1,\frac{(2\Lambda)^{{2-p}}}{\gamma_{1}}\right\}\gamma_{4}
    \,\textrm{exp}(2\gamma_{5}k\,M). 
  \end{equation}
  Observe that this choice is always possible since $p\leq2$.  
  
  On the other hand, in the calculations with $N>1$ we can simply
  use~\eqref{Gronwall_discrete2} (which starts at $N=1$
  if~\eqref{eq:gamma4} is satisfied) as inductive assumption to
  estimate the right-hand side of~\eqref{eq:inductive}.
  
  As a result of the choice of $\gamma_{4}$, in both cases it follows
  with the same algebraic manipulations
  of~\cite[Lemma~3.3]{bdr-7-5-num} that we can bound the right-hand
  side of~\eqref{eq:inductive} as follows:
  \begin{equation*}
    1<b_N^{2(p-1)}\leq \frac{h^{2}}{k}
    \frac{2\,{(1+\Lambda)^{2(2-p)}}}{\gamma_{1}} 
    \left[\gamma_0+\frac{\gamma_0}{\gamma_{1}}+\frac{\gamma_{2}^{2} 
        \gamma_{4}} {\gamma_{1} 
        (\lambda+\Lambda)^{p-2}} 
      \textrm{exp}(2\gamma_{5}k\,N)+\gamma_0
    \right]. 
  \end{equation*}
  This gives a contradiction, provided that 
  \begin{equation*}
    \frac{h^{2}}{k}\leq \frac {\gamma_{1}}{2\,{(1+\Lambda)^{2(2-p)}}}
    \left[\gamma_0+\frac{\gamma_0}{\gamma_{1}}+
      \frac{\gamma_{2}^{2}\gamma_{4}}{\gamma_{1}(\lambda+\Lambda)^{p-2}} 
      \textrm{exp}(2\gamma_{5}k\,N)+\gamma_0\right]^{-1}:=\overline{\gamma_0}.
  \end{equation*}
  This finally proves that $b_N\leq1$. Then the rest of the proof goes
  exactly as in the cited lemma, with a further application of the
  standard discrete Gronwall lemma. It is in this last step that one
  has to assume the limitation
  $k<\overline{k}:=\min\{1,(2\gamma_{5})^{-1}\}$.
\end{proof}

\medskip 

Inspecting the proof it is clear that we have also the following
result, where the doubling of the constant on the right-hand side comes
from having the same estimate separately for $a_{0},b_{0}$ and for
$a_{m},b_{m}$, for $m>1$.

\begin{Corollary}
  Let the same hypotheses of Lemma~\ref{lem:discrete_Gronwall_lemma}
  be satisfied, then in addition we have that
  \begin{align}
    &\max_{0\leq m\leq M} b_m\leq 1,
    \\
    &\max_{0\leq m\leq M} a_m^2+{\gamma_1(\lambda+\Lambda)^{p-2}}
    \ksum b_m^2\leq  2\gamma_4\,h^{2}\,\textrm{exp}(2 \gamma_{5}k\,M).
  \end{align}
\end{Corollary}
\subsection{Numerical algorithms}
We write explicitly the numerical algorithms we will consider and
state some basic existence results for the space-time-discrete solutions.

Given a net $I^M$, a triangulation $\mathcal{T}_h$ of $\Omega$, and
conforming spaces $V_h,\,Q_h$, (recall notation from
Sec.~\ref{sec:notation}) for the space-time-discrete problem, we use
the following algorithm:

\medskip

\noindent\textbf{Algorithm (space-time-discrete, Euler semi-implicit)} Let
$\bu^0_h=\Pidiv\bu_0$.  Then, for $m\geq1$ and $\bu^{m-1}_h\in V_h$
given from the previous time-step, compute the iterate
$(\bfumh,\pi^m_h)\in V_h\times Q_h$ such that for all $\bxi_h\in V_h$,
and $\eta_h\in Q_h$
\begin{equation}
  \label{eq:Qmh_iterate} 
  \tag{$Q^m_h$}
  \begin{aligned}
   \hspace*{-3mm} \langle d_t
    \bfumh,\bxi_h\rangle_h\!+\!\langle\bS(\bD\bfumh),\bD\bxi_h\rangle_h\! +\!
    b(\bu^{m-1}_h\!,\bfumh,\bxi_h)
       \! - \!\langle \divo\bxi_h,\pi^m_h\rangle_h
    &=\langle\ff(t_m),\bxi_h\rangle_h,\hspace*{-2mm}
    \\
    \langle\divo \bfumh,\eta_h\rangle_h&=0.
  \end{aligned}
\end{equation}
We also observe that the (space-continuous) time-discrete
scheme~\eqref{m_iterate} from Theorem~\ref{thm_regularity_discrete}
can be formulated in a weak form as follows: Let be given
$\bu^0=\bu_0$, $m\geq1$, and $\bu^{m-1}\in V$ evaluated from the
previous time-step, compute the iterate $(\bfum,\pi^m)\in V\times Q$
such that for all $\bxi\in V$,
and $\eta\in Q$
\begin{equation*}
  \tag{$Q^m$}
  \begin{aligned}
    \skp{d_t\bfum}{\bfxi}+\skp{\bS(\bD\bfum)}{\bD\bfxi}+b(\bfu^{m-1},
    \bfum,\bfxi)
    -\skp {\divo\bfxi}{\pi^m}
    &=\skp{\ff(t_m)}{\bfxi},
    \\
    \skp{\divo\bfum}{\eta}&=0.
  \end{aligned}
\end{equation*}
The existence of a solution $(\bfum,\pi^m)$ and its uniqueness follow
from Thm.~\ref{thm_regularity_discrete}, concerning strong solutions
$\bfum\in V(0)$ of~\eqref{m_iterate}. This solution is a fortiori also
a weak solution of the following problem: Find $\bfum\in V(0)$ such
that for all $\bxi\in V(0)$
\begin{equation*}
\tag{$P^m$}
  \begin{aligned}
    \langle d_t
    \bfum,\bxi\rangle+\langle\bS(\bD\bfum),\bD\bxi\rangle+b(\bu^{m-1},\bfum,\bxi) 
    &=\langle\ff(t_m),\bxi\rangle ,
  \end{aligned}
\end{equation*}
where $V(0):=\{\bw\in V:\ \skp{\divo\bw}{\eta}=0, \ \forall\,\eta\in
Y\}$. The existence of the associated pressure $\pi^m\in Q$ follows
then from the DeRham theorem and the inf-sup condition.

The situation for the space-time-discrete problem is similar: The existence
of the solution $(\bfumh,\pi^m_h)$ can be inferred in the following
way.  First, for $V_h(0)=\{\bw_h\in V_h:\ \skp{\divo\bw_h}{\eta_h}=0,
\ \forall\,\eta_h\in Y_h\}$ consider the following algorithm, given
$\bfu^{m-1}\in V_h(0)$ find $\bfum\in V_h(0)$ such that for all
$\bxi_h \in V(0)    $
\begin{equation}
  \label{eq:Pmh_iterate} \tag{$P^m_h$}
  \begin{aligned}
    \skp{d_t\bfumh}{\bxi_h}_h+\langle\bS(\bD\bfumh),\bD\bxi_h\rangle_h
    +b(\bu^{m-1}_h,\bfumh,\bxi_h)
    &=\langle\ff(t_m),\bxi_h\rangle_h\,.
  \end{aligned}
\end{equation}
The existence of a weak solution for~\eqref{eq:Pmh_iterate} follows
directly by applying the Brouwer's theorem, see
also~\cite[Lemma~7.1]{pr}. Uniqueness
follows from the semi-implicit expression for the convective term and
from the monotonicity of $\bS$. Moreover, the
following energy estimate holds true:
\begin{equation*}
  \max_{0\leq m\leq M}\|\bfumh\|^2_2+\ksum\|\bD\bfumh\|_p^p\leq C(\bfu_0,\bff),
\end{equation*}
which is obtained by using $\bfumh\in V_h(0)$ as test function.
Coming back to Problem~\eqref{eq:Qmh_iterate}, the existence of the
associated pressure $\pi^m_h\in Q_h$ such that $(\bfum_h,\pi^m_h) $ is
a solution of~\eqref{eq:Qmh_iterate} is derived from the previous
result of existence of a solution for $(P_h^m)$ and the inf-sup
condition. See also~\cite[Lemma~4.1]{bdr-phi-stokes} for such
inequality in the setting of Orlicz spaces.

\bigskip

\noindent The main result of this paper is the following error
estimate.
\begin{Theorem}
  \label{thm:main_thm}
  Let $\bS$ satisfy Assumption~\ref{ass:1} with
  $p\in\big(\frac{3}{2},2]$ and $\para \in [0,\para_0]$, where
  $\para_0>0$ is an arbitrary number. Let $ {\bff} \in {C(I;
    W^{1,2}(\Omega))} \cap {W^{1,2}(I; L^2(\Omega))}$, where
  $I=[0,T]$, for some $T>0$, and let ${\bfu_0} \in
  {W_{\divergence}^{2,2}(\Omega)}$ with $\divo \bS(\bD \bu_0) \in
  L^2(\Omega)$ be given. Let $\bfu $ be a strong
  solution of the (continuous) problem~\eqref{eq:pfluid} satisfying
  \begin{align}
    \label{eq:reg11}
    \begin{split}
      \norm{ \bfu_t}_{L^\infty(I; L^2(\Omega))} +
      \norm{\bF(\bD\bfu)}_{W^{1,2}(I\times\Omega)} +
      \norm{\bF(\bD\bfu)}_{L^{2\frac{5 p-6}{2-p}}(I;W^{1,2}(\Omega))}
      &\leq c_{2}.
    \end{split}
  \end{align}
  Let $\mathcal{T}_h$ be a triangulation as introduced in
  Sec.~\ref{sec:preliminaries} and let $(\bfumh,\pi^m_h)$ be the
  unique solution of the space-time-discrete problem~\eqref{eq:Qmh_iterate}
  corresponding to the data $(\Pidiv\bu_0,\bff)$.  Then, there exists
  a time-step $k_{1}>0$ and a mesh-size $h_{1}>0$ such that, if
  $\max\{h^{\frac{3p-2}{2}},h^{2}\}\leq c_{3}\, k$ for some $c_{3}>0$,
  for all $k \in (0,{k}_{1})$ and for all $h\in(0,h_{1})$, then
  the following error estimate holds true:
  \begin{equation*}
      \max_{0\leq m\leq M}\|\bfu(t_m)-\bfum_h\|^2_{2}+
      \ksum\|\bF(\bD\bfutm)-\bF(\bD\bfumh)\|^2_{2}\leq c_4\,
      (h^2+k^2).
  \end{equation*}
  The constants ${k}_{1},$ $h_{1}$, $c_3$ and $c_{4}$ depend only on
  $c_{2}$, the characteristics of $\bS$, and $|\Omega|$, but they are
  independent of $\para \in [0,\para_0]$.
\end{Theorem}
\begin{remark}{\rm 
  A explained in~\cite{bdr-7-5-num}, in the space periodic setting we
  are able to obtain~\eqref{eq:reg11} starting from the assumptions on
  the data of the problem, at least in a small time interval $[0,T']$,
  see Thm.~\ref{thm_existence_strong}. On the other hand, the analysis
  performed below is correct also in the Dirichlet case, provided one
  can show the regularities \eqref{eq:reg11} and
  \eqref{eq:regularity_discrete_problem} for the continuous and
  time-discrete problem, respectively. 

  In the case of Dirichlet data, the semigroup approach of Bothe
  and Pr{\"u}ss~\cite{pruess} proves the existence
  and uniqueness of a strong solution in a small interval $[0,T']$
  for $p\ge 1$, under the hypothesis of smooth data and $\delta>0$.
  Note that their analysis does not ensure the regularity  of the time
  derivative stated in~\eqref{eq:reg11}. However, one can easily prove
  this property based on the results proved in \cite{pruess} by
  standard techniques. 

  Unfortunately the needed space regularity is still open even for
  steady problems with 
  $(p,\delta)$-structure in the Dirichlet case, and this would be the
  basis for the regularity of the time-discrete problem. For partial
  results in the steady case, see for instance Beir{\~a}o da
  Veiga~\cite{hugo-p-less2}.

  We also wish to point out that one of the main difficulties in the
  Dirichlet case is that of having estimate independent of $\delta$,
  which is one of the key points in our analysis also of the
  degenerate problem.}
\end{remark}
\section{Proof of the main result}
\label{sec:proof-main-result}
The proof of Thm.~\ref{thm:main_thm} is obtained by splitting the
numerical error as follows:
\begin{equation*}
  \bu(t_m)-\bfumh=\bu(t_m)-\bfum+\bfum-\bfumh=:\boldsymbol{\epsilon}^m+\bfem, 
\end{equation*}
For the error $\boldsymbol{\epsilon}^m$
Thm.~\ref{thm_regularity_discrete} ensures 
\begin{equation*}
  \max_{0\leq m\le M }\|
  \,\bfutm-\bfum\|^{2}_{2}+\ksum\|\bF(\bD\bfutm)-\bF(\bD\bfum)\|^2_{2} 
  \leq c\,k^{2}.   
\end{equation*}
Hence, we need to focus only on the second part of the error, namely
$\bfem$.

The main error estimate is obtained by taking the difference between
the equation satisfied by $\bfum$ and that for $\bfumh$, and using as
test function $\bfxi_h\in V_h\subset V$. In this way we obtain the
following error equation for all $\bfxi_h\in V_h$
\begin{equation}
  \label{eq:error_equation}
  \begin{aligned}
    \skp{d_t\bfem}{\bfxi_h}_h&+\skp{\bS(\bD\bfum)-\bS(\bD\bfumh)}{\bD\bfxi_h}_h
    +b(\bu^{m-1},\bfum,\bfxi_h) 
    \\
    &-b(\bu^{m-1}_h,\bfumh,\bfxi_h)-\skp{\divo
      \bfxi_h}{\pi^m-\pi^m_h}_h=0.
  \end{aligned}
\end{equation}
Clearly, a ``natural'' test function $\bxi_h$ to get the error
estimate would be $\bfem:=\bfum-\bfumh$, which cannot be used, since
it is not a discrete functions, that is $\bfem\not\in V_h$.  The error
estimate is then obtained by using as test function the projection
$\bfxi_h:=\Pidiv\bfem\in V_h$ and treating the various terms arising
from the following identity:
\begin{align*}
  \Pidiv\bfem=\Pidiv(\bfum-\bfumh)=\Pidiv\bfum-\bfumh&=
  \Pidiv\bfum-\bfum+\bfum-\bfumh\\
  &=:    \bR^m_h+\bfem,
\end{align*}
where we used that $\Pidiv=\textrm{id}$ on $V_h$.
Let us start from the first term from the left-hand side of the error
equation, that one concerning  the discrete time-derivative. We
have the following result
\begin{lemma}
  \label{lem:lemma1}
  The following estimate holds true 
  \begin{equation*}
    \frac{1}{2}d_t\|\bfem\|^2_2+\frac{k}{4}\|d_t\bfem\|^2_2-
    \frac{1}{k}\|\bR^m_h\|^2_2\leq
    \skp{d_t\bfem}{\Pidiv\bfem}_h.
  \end{equation*}
\end{lemma}
\begin{proof}
  By standard manipulations of the discrete
  time-derivative  we get 
  \begin{equation*}
    \begin{aligned}
      \skp{d_t \bfem}{\Pidiv\bfem}_h & = \skp{d_t \bfem}{\bfem}_h+
      \skp{d_t\bfem}{\bR^{m}_{h}}_h
      \\
      &=\frac{1}{2}d_t\|\bfem\|^2_2+\frac{k}{2}\|d_t\bfem\|^2_2+
      \skp{d_t\bfem}{\bR^{m}_{h}}_h. 
    \end{aligned}
  \end{equation*}
  Observe now that, by Young's inequality, we have
  \begin{equation*}
    \big|  \skp{d_t\bfem}{\bR^{m}_{h}}\big|
    \leq\frac{ k}{4}\|d_t\bfem\|^2_2+\frac{1}{k}\|\bR^m_h\|^2_2,
  \end{equation*}
  hence the statement.
\end{proof}

Next, we treat the second term from the left-hand side of the error
equation, that one related with the extra stress-tensor.
\begin{lemma}
  \label{lem:lemma2}
  There exists $c>0$, independent of $h$ and $\delta$, such that
  \begin{equation*}
    \begin{aligned}
      c\,\big(\|\bF(\bD\bfum)-\bF(\bD\bfumh)\|^2_2-&\|\bF(\bD\bfum)-\bF(\bD
      \Pidiv\bfum)\|^2_2\big)
      \\
      &\qquad \leq    \skp{\bS(\bD\bfum)-\bS(\bD\bfumh)}{\bD(\Pidiv\bfem)}_h.
    \end{aligned}
  \end{equation*}
\end{lemma}
\begin{proof}
  From  standard properties concerning the structure of $\bS$, as
  recalled in Sec.~\ref{sec:stress_tensor}, we get
  \begin{equation*}
    \begin{aligned}
      & \skp{\bS(\bD\bfum)-\bS(\bD\bfumh)}{\bD(\Pidiv\bfum-\bfumh)}_h
      \\
      &= \skp{\bS(\bD\bfum)-\bS(\bD\bfumh)}{\bD\bfum-\bD\bfumh}_h
      + \skp{\bS(\bD\bfum)-\bS(\bD\bfumh)}{\bD \Pidiv\bfum-\bD\bfum}_h
      \\
      &=\|\bF(\bD\bfum)-\bF(\bD\bfumh)\|^2_2+
      \skp{\bS(\bD\bfum)-\bS(\bD\bfumh)}{\bD( \Pidiv\bfum-\bfum)}_h.
    \end{aligned}
  \end{equation*}
  The latter term on the right-hand side can be estimated with the
  help of Lemma~\ref{lem:quasinormtrick} as
  follows
  \begin{equation*}
    \begin{aligned}
      &\big|\skp{\bS(\bD\bfum)-\bS(\bD\bfumh)}{\bD(
        \Pidiv\bfum-\bfum)}_h\big|
      \\
      &\qquad\leq\epsilon
      \|\bF(\bD\bfum)-\bF(\bD\bfumh)\|^2_2+c_\epsilon\|\bF(\bD\bfum)-\bF(\bD
      \Pidiv\bfum)\|^2_2,
    \end{aligned}
  \end{equation*}
  ending the proof.
\end{proof}

Some care is needed also to handle the two terms coming from the
convective term, which are estimated in the following lemma, by using
the regularity results for the solution $\bfum$ of the time-discrete
problem.
\begin{lemma}
  \label{lem:lemma3}  
  There exist $c>0$ and $\theta\in]0,1[$, not depending on $h$ and
  $\parameter$, such that
  \begin{equation}
    \label{eq:estimate_convective}
    \begin{aligned}
      &\big| b(\bu^{m-1},\bfum,\Pidiv\bfem)
      -b(\bu^{m-1}_h,\bfumh,\Pidiv\bfem)\big|
      \\
      &\quad \leq c\,\big( \|\nabla\bR_h^m\|_{\frac{3p}{3+1}}\|\bD\bfem\|_{p}+
      \|\be^{m-1}\|_{2}^\theta\|\bD\be^{m-1}\|_{p}^{1-\theta}\|\bD\bfem\|_p\big).
    \end{aligned}
  \end{equation}
\end{lemma}
\begin{proof}
  By adding and subtracting $b(\bu^{m-1},\Pidiv\bfum,\Pidiv\bfem)$
  and also in a second step $b(\bu^{m-1}_h,\Pidiv\bfum,\Pidiv\bfem)$
  and, by observing that $b(\bu^{m-1}_h,\Pidiv\bfem,\Pidiv\bfem)=0$,
  we get
  \begin{equation*}
    \begin{aligned}
      & b(\bu^{m-1},\bfum,\Pidiv\bfem) -b(\bu^{m-1}_h,\bfumh,\Pidiv\bfem)
      \\
      & = b(\bu^{m-1},\bfum-\Pidiv\bfum,\Pidiv\bfem)
      +b(\bu^{m-1}_{h},\Pidiv(\bfum-\bfum_{h}),\Pidiv\bfem)
      \\
      &\qquad+b(\bu^{m-1}-\bu^{m-1}_h,\Pidiv\bfum,\Pidiv\bfem)
      \\
      &
      =b(\bu^{m-1},\bR^{m}_{h},\Pidiv\bfem)+b(\be^{m-1},\Pidiv\bfum,\Pidiv\bfem)
      \\
      & =:I_1+I_2.
    \end{aligned}
  \end{equation*}
  Since $\divo\bu^{m-1}=0$ (in the continuous sense) the first term is
  estimated as follows, by using H{\"o}lder inequality
  \begin{equation*}
    \begin{aligned}
      I_1&=b(\bu^{m-1},\bR_h^m,\Pidiv\bfem)=\skp{[\nabla
        \bR_h^m]\bu^{m-1}}{\Pidiv\bfem}_{h} 
      \\
      &
      \leq\|\bu^{m-1}\|_{\frac{3p}{3p-4}}\|\nabla\bR_h^m\|_{\frac{3p}{p+1}}
      \|\Pidiv\bfem\|_{\frac{3p}{3-p}}
    \end{aligned}
  \end{equation*}
  provided that $p > \frac 43$. By a Sobolev embedding theorem, the
  Korn's inequality (valid in the case of functions vanishing at the
  boundary or with zero mean value), and by the continuity of the
  interpolation operator $\Pidiv$ (cf.~Prop.~\ref{thm:ocont},
  Rem.~\ref{rem:int}) we can write
  \begin{equation*}
    \|\Pidiv\bfem\|_{\frac{3p}{3-p}}\leq c\,\|\nabla
    \Pidiv\bfem\|_p\leq c\,\|\bD\bfem\|_{p}.
  \end{equation*}
  Thus we arrive at $I_1\le
  c\,\|\bu^{m-1}\|_\infty\|\nabla\bR_h^m\|_{\frac{3p}{3+1}}\|\bD\bfem\|_{p}$
  and now observe that, by using
  regularity~\eqref{eq:regularity_discrete_problem} of the solution
  $\bu^m$, this term is bounded by the first one from the right-hand
  side of~\eqref{eq:estimate_convective}.

  Concerning $I_2$, by using the definition of $b(\,.\,,\,.\,\,,.\,)$,
  we split it as follows
  \begin{equation*}  
    I_2=I_{2,1}+I_{2,2}:=\frac{1}{2}\skp{[\nabla\Pidiv\bfem]\be^{m-1}}{\Pidiv\bfum}_h-
    \frac{1}{2}\skp{[\nabla\Pidiv\bfum]\be^{m-1}}{\Pidiv\bfem}_{h}.
  \end{equation*}
  We estimate $I_{2,1}$ with the H{\"o}lder inequality:
  \begin{equation*}
    I_{2,1}\leq
    c\,\|\Pidiv\bfum\|_{\alpha}\|\be^{m-1}\|_{s_{1}}\|\nabla\Pidiv\bfem\|_p, 
  \end{equation*}
  for some $s_{1}\in (p',p^*)=\big (\frac{p}{p-1},\frac{3p}{3-p}\big )$ and
  $\alpha=\frac{ps_{1}}{ps_{1}-s_{1}-p}<\infty$. We have that
  $2\leq p'$, hence, by standard
  convex interpolation with $\theta\in (0,1)$ such that
  $\frac{1}{s_{1}}=\frac{\theta}{2}+\frac{(1-\theta)}{p^{*}}$, by the
  properties of $\Pidiv$ (cf.~Prop.~\ref{thm:ocont},
  Rem.~\ref{rem:int}), by Korn's inequality, and
  since~\eqref{eq:regularity_discrete_problem} implies
  $\|\bfum\|_\alpha \le c\,\|\bfum\|_{{\infty}}\in l^{\infty}(I_M)$, we obtain that
  \begin{equation*}
    \begin{aligned}
      I_{2,1}&\leq c\,\|\bfum\|_{\alpha}\|\be^{m-1}\|_{2}^\theta
      \|\bD\be^{m-1}\|_{p}^{1-\theta}\|\nabla\Pidiv\bfem\|_p
      \\
      &\leq
      c\,\|\be^{m-1}\|_{2}^\theta\|\bD\be^{m-1}\|_{p}^{1-\theta}\|\bD\bfem\|_p.
    \end{aligned}
  \end{equation*}
For the term $I_{2,2}$ we have, by H{\"o}lder inequality
\begin{equation*}
    I_{2,2}\leq\frac{1}{2}\|\nabla \Pidiv\bfum\|_{r}\|\be^{m-1}\|_{s_{2}}\|
    \Pidiv\bfem\|_{\frac{3p}{3-p}},
\end{equation*}
for some $1<r<6(p-1)$ and $s_{2}=\frac{r p^{*}}{r p^{*}-r-p^{*}}$. A
straightforward computation shows that for any $\frac{3}{2}<p\leq2$
one can choose $r$ close enough to $6(p-1)$ in such a way that
$s_{2}<\frac{p}{p-1}<s_{1}$. Hence, by using again the properties of
the interpolation operator (cf.~Prop.~\ref{thm:ocont},
  Rem.~\ref{rem:int}), since
by~\eqref{eq:regularity_discrete_problem} we have that
$\|\nabla \bfum\|_{r}\in l^{\infty}(I_M)$ for all $r<6(p-1)$, by H{\"o}lder
and Korn's inequality, and by the embedding $L^{s_{1}}(\Omega)\subset
L^{s_{2}}(\Omega)$, we get
\begin{equation*}
    I_{2,2}\leq c\,\|\be^{m-1}\|_{s_{1}}\|\bD\bfem\|_{p}.
\end{equation*}
Thus the right-hand can be estimated as $I_{2,1}$, which completes the proof. 
\end{proof}

\medskip

Now we need to estimate the term involving the pressure which can be
handled by using the same approach as in~\cite{bdr-phi-stokes}. Note
that the regularity for the gradient of the pressure represents an
outstanding open problem, with only partial results. In fact, at
present, for the time evolution (either continuous or discrete) there
are only results which exclude the degenerate case $\delta=0$,
see~\cite{bdr-7-5-num,bdr-7-5}.  Let us now show how the last term
in~\eqref{eq:error_equation} is estimated only in terms of the
external force and of the velocity, by using once again the equations,
as done in~\cite{bdr-phi-stokes}.
\begin{lemma}
  \label{lem:lemma4}
  For each $\epsilon>0$ there exists $c_{\epsilon}>0$, not depending on $h$ and
  $\parameter$, such that
  \begin{equation*}
    \begin{aligned}
    &  \big| \skp {\divo \Pidiv\bfem}{\pi^m-\pi^m_h}_{h}\big|
      \\
      &\qquad \leq c\, \sum_K \int_K
      \big(\phi_{|\bD\bfu^m|}\big)^*\big(h|\bff(t_m)|+h|d_t\bfum|+ 
      h|\bu^{m-1}|\,       |\nabla\bfum|\big)\,dx
      \\
      &\qquad \quad + c\, \sum_K\int_{S_K} \!\!\abs{\bF(\bD\bfum)-\mean{\bF(\bfD
          \bfum)}_{S_K}}^2\,dx
      \\
      &\qquad  \quad +\epsilon\,\Big(\,\norm{ \bfF( \bfD\bfum)-\bfF(\bfD
        \Pidiv\bfum)}^2_2 +\norm{\bfF(\bfD \bfum) - \bfF(\bfD
        \bfumh)}^2_2 \Big).
    \end{aligned}
  \end{equation*}
\end{lemma}
\begin{proof}
  We start by observing that for all $\bfxi_h\in V_h(0)$ we have
  \begin{equation*}
    \skp{\divo\bfxi_h}{\pi^m-\pi^m_h}_{h}=\skp{\divo
      \bfxi_h}{\pi^m-\eta^m_h}_{h}\qquad\forall\, \eta^m_h\in Y_h.
  \end{equation*}
  Then, if we use (in the same way as
  in~\cite[Lemma~3.1]{bdr-phi-stokes}) the divergence-preserving
  projection operator $\Pidiv$, we can estimate the term
  involving the pressure in the error equation as follows: For each
  $\eta^m_h\in Y_h$ it holds
  \begin{equation*}
    \begin{aligned}
      & \abs{\skp{\divo\Pidiv\bfem}{\pi^m-\pi^m_h}_{h}}=
      \abs{\skp{\divo(\Pidiv\bfum-\bfumh)}{\pi^m-\eta^m_h}_{h}}
      \\
      & \quad\leq \int_\Omega \abs{\bD \Pidiv\bfum -
        \bD\bfum+\bD\bfum-\bD\bfumh}\, \abs{\pi^m- \eta^m_h}\,dx
      \\
      & \quad\leq \epsilon\int_\Omega \phi_{|\mathbf{Du}^{m}|}(\abs{\bD
        \Pidiv\bfum - \bD\bfum}) + \phi_{|\mathbf{Du}^m|}(\abs{\bfD
        \bfum- \bfD\bfumh})\,dx
      \\
      &\hspace{1cm}  \quad+ c_\epsilon
      \int_\Omega(\phi_{|\mathbf{Du}^{m}|})^\ast
      (\abs{\pi^m-\eta^m_h})\,dx
      \\
      & \quad\leq \epsilon\,c\,\Big(\,\norm{ \bfF( \bfD\bfum)-\bfF(\bfD
        \Pidiv\bfum)}^2_2 +\norm{\bfF(\bfD \bfum) - \bfF(\bfD
        \bfumh)}^2_2 \Big)
      \\
      &\hspace{1cm} \quad +c_\epsilon\int_\Omega
      (\phi_{|\mathbf{Du}^{m}|})^\ast (\abs{\pi^m- \eta^m_h})\,dx.
    \end{aligned}
  \end{equation*}
  In particular we can choose $\eta_h^m=\PiY\pi^m$. By using also
  Assumption~\ref{ass:PiY} the latter term is estimated by using the
  same techniques as in~\cite[Lemma~6.4]{bdr-phi-stokes} as follows:
  \begin{equation*}
    \begin{aligned}
      \int_K \!(\phi_{\abs{\bfD \bfv}})^\ast (\abs{\pi^m- \PiY
        \pi^m})\,dx&\leq c\! \int_K \!\!\big(\phi_{|\bD
        \bv|}\big)^*(h|\bff(t_m)|+h|d_t\bfum|+h|\bu^{m-1}|\,
      |\nabla\bfum|)\,dx
      \\
      &\quad+ c \int_{S_K} \!\!\abs{\bF(\bD\bv)-\mean{\bF(\bfD
          \bfv)}_{S_K}}^2\,dx.
    \end{aligned}
  \end{equation*}
Finally,  summing over $K\in \mathcal{T}_h$ we get the assertion.
\end{proof}

\medskip 

By collecting the above results we can now prove the main result of
the paper.

\begin{proof}[{\it of Theorem~\ref{thm:main_thm}}]
  By gathering the results from
  Lemmas~\ref{lem:lemma1}-\ref{lem:lemma4} we get the following
  discrete inequality: exists $c>0$, independent of $\delta$ and $h$,
  and $\theta\in (0,1)$ such that
  \begin{equation*}
    \begin{aligned}
      d_t\|\bfem\|^2_2+&k\|d_t\bfem\|^2_2+\|\bF(\bD\bfum)-\bF(\bD\bfumh)\|^2_2+
      (\parameter+\|\bD\bfem\|_p)^{p-2}\|\bD\bfem\|_p^2
      \\
      &\leq c\, \Big[\,\frac{1}{k}\|\bR^m_h\|^2_2+\|\bF(\bD\bfum)-\bF(\bD
      \Pidiv\bfum)\|^2_2+\|\nabla
      \bR_h^m\|_{\frac{3p}{3+1}}\|\bD\bfem\|_{p}
      \\
      &\quad+\|\be^{m-1}\|_{2}^\theta\|\bD\be^{m-1}\|_{p}^{1-\theta}\|\bD\bfem\|_p
      \\
      &\quad+ \sum_K \int_K \!\!\big(\phi_{|\bD
        \bv|}\big)^*(h|\bff(t_m)|+h|d_t\bfum|+h|\bu^{m-1}|\,
      |\nabla\bfum|)\,dx
      \\
      &\quad+ \sum_K\int_{S_K} \!\!\abs{\bF(\bD\bfum)-\mean{\bF(\bfD
          \bfum)}_{S_K}}^2\,dx\Big],
  \end{aligned}
\end{equation*}
and by using a Sobolev embedding we can also obtain the following
bound
\begin{equation*}
  \|\be^{m-1}\|_{2}^\theta\|\bD\be^{m-1}\|_{p}^{1-\theta}\|\bD\bfem\|_p\leq
  c\,\|\bD\be^{m-1}\|_{p}\|\bD\bfem\|_p.
\end{equation*}
With this observation and by setting
\begin{equation*}
  \begin{aligned}
     a_m&:=\|\bfem\|_2,
    \\
    b_m&:=\|\bD\bfem\|_p,
    \\
    r_m&:=\|\nabla\bR_h^m\|_{\frac{3p}{p+1}},
    \\
    s_m^2&:= \|\bF(\bD\bfum)-\bF(\bD \Pidiv\bfum)\|^2_2+\sum_K
    \int_{S_K} \abs{\bF(\bD\bfum)-\mean{\bF(\bfD
        \bfum)}_{S_K}}^2\,dx
\\
&\qquad +  \sum_K \int_K \!\!\big(\phi_{|\bD
        \bv|}\big)^*(h|\bff(t_m)|+h|d_t\bfum|+h|\bu^{m-1}|\,
      |\nabla\bfum|)\,dx  +\frac{\|\bR^m_h\|_2^2}{k},
  \end{aligned}
\end{equation*}
we have that the two inequalities \eqref{discrete_Gronwall_bis},
\eqref{discrete_Gronwall_ter} are satisfied.  Hence, in order to apply
Lemma~\ref{lem:discrete_Gronwall_lemma}, we need just to verify the
hypotheses on the initial values $a_{0},b_{0}$ and on $r_{m}$ and
$s_{m}$.

To this end, first we observe that $\bfe^0=\bu_0-\Pidiv \bu_0$.  By
using the assumption $\bu_0\in W^{2,2}_{\divo}$, by the properties of
the interpolation operator $\Pidiv$, and due to $p\leq2$ we obtain:
\begin{equation}
  \label{eq:initial_data}
  \|\bfe^0\|_{2}\leq c\, h^2\qquad
  \text{and}\qquad\|\bD\bfe^0\|_{p}\leq c\,h
\end{equation}
We now check the hypotheses needed on $r_{m}$ and we observe, that if
$\bfum\in W^{2,\frac{3p}{p+1}}(\Omega)$, then
\begin{equation*}
  \|\nabla\bR_h^m\|_{\frac{3p}{p+1}}\leq
  c\,h\,\|\nabla^2\bfum\|_{\frac{3p}{p+1}}, 
\end{equation*}
by the properties of the interpolation operator (cf.~Prop.~\ref{thm:ocont},
  Rem.~\ref{rem:int}). Hence, under the
assumptions of regularity of $\bfum$, we also obtain that
\begin{equation*}
  \ksum   \|\nabla\bR_h^m\|_{\frac{3p}{p+1}}^2\leq c\,h^2,
\end{equation*}
for some constant $c$ independent of $\parameter$ and $h$. 

Let us now consider $s_{m}$ and we recall that if $\bF(\bD\bfum)\in
W^{1,2}(\Omega)$, then uniformly with respect to $K\in\mathcal{T}_h$
(cf.~\cite[Thm~3.7,Thm~5.1]{bdr-phi-stokes})
\begin{equation*}
  \begin{aligned}
    \|\bF(\bD\bfum)-\bF(\bD \Pidiv\bfum)\|^2_2&\leq \sum_K \int_{S_K}
    \abs{\bF(\bD\bfum)-\mean{\bF(\bfD \bfum)}_{S_K}}^2\,dx
    \\
    &\leq c\,h^2\|\nabla \bF(\bD\bfum)\|^2_2.
  \end{aligned}
\end{equation*}
We now estimate the third term in the definition of $s_{m}^{2}$ by
defining the following non-negative sequence $\{g^m\}_m$
\begin{equation*}
  g^m:=|\bff(t_m)|+|d_t\bfum|+|\bu^{m-1}|\,  |\nabla\bfum|.
\end{equation*}
By Young's inequality and by using the following inequality for
$\varphi$ defined in~\eqref{eq:5}
 \begin{equation*}
   (\phi_{a})^{*}(\kappa \,t)\leq c\,\kappa^{2}\,(\phi_{a})^{*}(t)\quad 
 \end{equation*}
 valid for $\kappa \in [0,\kappa_0]$ and $p\le 2$ with a constant $c$
 independent of $\delta,a$, and $t$ (cf.~\cite{bdr-phi-stokes}), we
 have
\begin{equation*}
  \begin{aligned}
    \sum_K \int_{S_K} \big(\phi_{|\mathbf{Du}^{m}|}\big)^*(h\, g^m)\,dx &\leq
    c\,h^2\sum_{K} \int_{S_K} \big(\phi_{|\mathbf{Du}^{m}|}\big)^*(g^m)\,dx
    \\
    & \leq
    c\,h^2\sum_K\int_{S_K} \phi(|\bD\bfum|)+\phi^*(g^m)\,dx.
  \end{aligned}
\end{equation*}
Pointing out that 
\begin{equation*}
  \ksum  \int_\Omega\phi^*(g^m)\,dx\leq
  \ksum\|g^m\|_{p'}^{p'}+(\parameter|\Omega|)^{p'}, 
\end{equation*}
we need just to check that $g^m\in l^{p'}(I^M)$. This follows by
interpolation from Thm.~\ref{thm_regularity_discrete}, and especially
from the observation in~\eqref{eq:regularity_time-discrete-derivative}

To conclude we need also to estimate the term $k^{-1}\|\bR_h^m\|^2_2$.
There is another (we also have one in the discrete Gronwall
Lemma~\ref{lem:discrete_Gronwall_lemma}) $h$-$k$ coupling that enters
the proof at this point. In fact, by Sobolev embedding, the standard
properties of interpolation operators in Sobolev space (see
e.g.~\cite[Thm.~3.1.6]{ciarlet}), and the assumptions on
$\mathcal{T}_h$ we get $ \|\bR_h^m\|_2\leq
c\,h^{\frac{5p-2}{2p}}\|\nabla^2 \bfum\|_{\frac{3p}{p+1}}$. Then, by
using the regularity on $\bfum$ from
Thm.~\ref{thm_regularity_discrete} we obtain
\begin{equation*}
  \ksum\frac{\|\bR^m_h\|_2^2}{k}\leq \frac{h^\frac{5p-2}{p}}{k}\ksum \|\nabla^2
  \bfum\|_{\frac{3p}{p+1}}^2\leq c \frac{h^\frac{5p-2}{p}}{k},
\end{equation*}
and if $h^{\frac{3p-2}{p}}\leq c\,k$, then
\begin{equation*}
  \ksum\frac{\|\bR^m_h\|_2^2}{k}\leq c\,h^{2}
\end{equation*}
This coupling between $k$ and $h$ derives from the natural regularity
of the problem, which is at the moment at disposal under rather
general assumptions on the data. We believe that this condition,
appearing also in simpler parabolic problems with
$p$-structure~\cite{der}, is only of technical character.

Then, by collecting all the previous estimate, we obtain that all the
hypotheses of Lemma~\ref{lem:discrete_Gronwall_lemma} are satisfied,
hence we end the proof.
\end{proof}

\section*{Acknowledgment}
L.Berselli would like to thank the SFB/TR 71
"Geometric Partial Differential Equations" for the hospitality during
his stays In Freiburg. L.Diening and M.\Ruzicka{} have been supported by by the
project C2 of the SFB/TR 71 "Geometric Partial Differential
Equations".

\def\cprime{$'$} \def\cprime{$'$} \def\cprime{$'$}
  \def\polhk#1{\setbox0=\hbox{#1}{\ooalign{\hidewidth
  \lower1.5ex\hbox{`}\hidewidth\crcr\unhbox0}}}
  \def\ocirc#1{\ifmmode\setbox0=\hbox{$#1$}\dimen0=\ht0 \advance\dimen0
  by1pt\rlap{\hbox to\wd0{\hss\raise\dimen0
  \hbox{\hskip.2em$\scriptscriptstyle\circ$}\hss}}#1\else {\accent"17 #1}\fi}
  \def\ocirc#1{\ifmmode\setbox0=\hbox{$#1$}\dimen0=\ht0 \advance\dimen0
  by1pt\rlap{\hbox to\wd0{\hss\raise\dimen0
  \hbox{\hskip.2em$\scriptscriptstyle\circ$}\hss}}#1\else {\accent"17 #1}\fi}
  \def\ocirc#1{\ifmmode\setbox0=\hbox{$#1$}\dimen0=\ht0 \advance\dimen0
  by1pt\rlap{\hbox to\wd0{\hss\raise\dimen0
  \hbox{\hskip.2em$\scriptscriptstyle\circ$}\hss}}#1\else {\accent"17 #1}\fi}
  \def\ocirc#1{\ifmmode\setbox0=\hbox{$#1$}\dimen0=\ht0 \advance\dimen0
  by1pt\rlap{\hbox to\wd0{\hss\raise\dimen0
  \hbox{\hskip.2em$\scriptscriptstyle\circ$}\hss}}#1\else {\accent"17 #1}\fi}
  \def\cprime{$'$}

\end{document}